\documentclass[11pt]{article}

\usepackage{amsmath}
\usepackage{amsfonts}
\usepackage{color,bm}
\usepackage{amssymb}
\usepackage{graphicx,nicefrac}
\usepackage{hyperref}
\usepackage{enumerate}
\usepackage[all]{xy}

\def\tr{{\rm tr}}

 \allowdisplaybreaks

\begin{document}

\newtheorem{problem}{Problem}

\newtheorem{theorem}{Theorem}[section]
\newtheorem{corollary}[theorem]{Corollary}
\newtheorem{definition}[theorem]{Definition}
\newtheorem{conjecture}[theorem]{Conjecture}
\newtheorem{question}[theorem]{Question}
\newtheorem{lemma}[theorem]{Lemma}
\newtheorem{proposition}[theorem]{Proposition}
\newtheorem{quest}[theorem]{Question}
\newtheorem{example}[theorem]{Example}

\newenvironment{proof}{\noindent {\bf
Proof.}}{\rule{2mm}{2mm}\par\medskip}

\newenvironment{proofof3}{\noindent {\bf
Proof of  Theorem 1.2.}}{\rule{2mm}{2mm}\par\medskip}

\newenvironment{proofof5}{\noindent {\bf
Proof of  Theorem 1.3.}}{\rule{2mm}{2mm}\par\medskip}

\newcommand{\remark}{\medskip\par\noindent {\bf Remark.~~}}
\newcommand{\pp}{{\it p.}}
\newcommand{\de}{\em}

\title{  {Improvements on some partial trace inequalities for positive semidefinite block matrices}\thanks{This paper was firstly announced in November 2021, 
and was later published on Linear and Multilinear Algebra, 2022. 
This is a final version;   
see \url{https://doi.org/10.1080/03081087.2022.2121368}.  E-mail addresses:
ytli0921@hnu.edu.cn (Y\v{o}ngt\={a}o L\v{i}).} }

\author{Yongtao Li$^{*}$\\
{\small School of Mathematics, Hunan University} \\
{\small Changsha, Hunan, 410082, P.R. China } 
 }

\maketitle

\vspace{-1cm}

\begin{center}
Dedicated to Prof. Weijun Liu  on his 60th birthday
 \end{center}

\begin{abstract}
We study matrix inequalities involving partial traces for positive semidefinite 
block matrices. 
First of all, we present a new method 
to prove a celebrated result of Choi [Linear Algebra Appl. 516 (2017)]. 
The method also allows us to prove a generalization of 
another result of Choi [Linear Multilinear Algebra 66 (2018)]. 
Furthermore, we shall give an improvement on 
a recent result of Li, Liu and Huang [Operators and Matrices 15 (2021)]. 
In addition, we include with some majorization inequalities  involving 
partial traces for two by two block matrices, and also provide inequalities related to 
the unitarily invariant norms as well as the singular values, 
which can be viewed as slight extensions of two results of Lin [Linear Algebra Appl. 459 (2014)] and [Electronic J. Linear Algebra 31 (2016)]. 
 \end{abstract}



\section{Introduction}

\label{sec1}

The space of $m\times n$ complex matrices is denoted by $\mathbb{M}_{m\times n}$; 
if $m=n$, we write $\mathbb{M}_n$ for $\mathbb{M}_{n\times n}$ 
and if $n=1$, we use $\mathbb{C}^m$ for $\mathbb{M}_{m\times 1}$. 
By convention, if $A\in \mathbb{M}_n$ is positive semidefinite, 
we write $A\ge 0$. For  Hermitian matrices $A$ and $B$ with the same size, 
$A\ge B$ means that  $A-B$ is positive semidefinite, i.e., $A-B\ge 0$. 
 We denote by $\mathbb{M}_m(\mathbb{M}_n)$ the set of $m\times m$ block matrices 
with each block in $\mathbb{M}_n$. 
Each element of $\mathbb{M}_m(\mathbb{M}_n)$ is also viewed as an $mn\times mn$ matrix 
with numerical entries and  usually written as $A=[A_{i,j}]_{i,j=1}^m$,  
where $A_{ij}\in \mathbb{M}_n$.   
We denote by $A\otimes B$ the Kronecker product of $A$ with $B$, that is, 
if $A=[a_{i,j}]\in \mathbb{M}_m$ and $B\in \mathbb{M}_n$, 
then $A\otimes B \in \mathbb{M}_m(\mathbb{M}_n)$ whose $(i,j)$ block is $a_{i,j}B$.

In the present paper, we are mainly concentrated on the block positive semidefinite matrices. 
Given  $A=[A_{i,j}]_{i,j=1}^m\in \mathbb{M}_m(\mathbb{M}_n)$, 
 the  partial transpose of $A$ is defined as 
\[ A^{\tau} = [A_{j,i}]_{i,j=1}^m=\begin{bmatrix} 
A_{1,1} & \cdots & A_{m,1} \\ 
\vdots & \ddots & \vdots \\
A_{1,m} & \cdots & A_{m,m}
\end{bmatrix}. \]
This is different from the usual transpose, which is defined as 
 \[ A^{T} = [A_{j,i}^T]_{i,j=1}^m=\begin{bmatrix} 
A_{1,1}^T & \cdots & A_{m,1}^T \\ 
\vdots & \ddots & \vdots \\
A_{1,m}^T & \cdots & A_{m,m}^T
\end{bmatrix}. \]
The primary application of partial transpose is materialized 
in quantum information theory \cite{Horo96, Hiro03, Petz08}. 
It is easy to see that $A\ge 0$ does not necessarily imply $A^{\tau}\ge 0$. 
If both $A$ and $A^{\tau}$ are positive semidefinite, 
then $A$ is said to be positive partial transpose (or PPT for short). 
 For more explanations and applications of the PPT matrices, 
we recommend a comprehensive monograph \cite{Bha07}, 
and see, e.g.,  \cite{Lin14,Lee15,Zhang19,FLT2021CMB,FLT2021} for recent results. 
Next we introduce the definition of two partial traces $\tr_1 A$ and $\tr_2 A$ of $A$. 
 \[  \tr_1 A = \sum_{i=1}^m A_{i,i} ~~~\text{and}~~~\tr_2 A = [\tr A_{i,j}]_{i,j=1}^m, \]
where $\tr (\cdot )$ stands for the usual trace. 
Clearly, we have $\tr_1 A \in \mathbb{M}_n$ and $\tr_2 A\in \mathbb{M}_m$.

 It is believed that 
 there are many elegant matrix inequalities that 
 have arisen from  the probability theory and  quantum information theory 
 in the literature.  
As we all know, these two partial trace maps are linear and trace-preserving. 
Furthermore, if ${ A}=[A_{i,j}]_{i,j=1}^m \in \mathbb{M}_m(\mathbb{M}_n)$ is positive semidefinite,  
it is easy to see that  both $\mathrm{tr}_1 A$ and $\mathrm{tr}_2 A$ 
are also positive semidefinite; see, e.g., \cite[p. 237]{Zhang11} or \cite[Theorem 2.1]{Zha12}. 
Over the years, various results involving partial transpose and partial traces have been obtained 
in the literature, e.g., \cite{Ando14,Choi17,Choi18,FLTmia,HL20}. 
We  introduce the background and recent progress  briefly. 
In the following results, 
we always assume that  $A\in \mathbb{M}_m(\mathbb{M}_n)$ 
is a positive semidefinite block matrix.   
\begin{itemize}
\item[(R1)]  
Choi \cite{Choi17} presented  
by using induction on $m$ that $I_m\otimes \tr_1 A^{\tau} \ge A^{\tau}$.  

\item[(R2)]
Zhang \cite{Zhang19} revisited Choi's result and proved $(\tr_2 A^{\tau})\otimes I_n\ge A^{\tau}$. 

\item[(R3)] 
Choi \cite{Choi18} further extended the result of Zhang  and proved 
\begin{equation} \label{eqr5}
I_m\otimes \tr_1 A^{\tau} \ge \pm A^{\tau}, 
\end{equation}
and 
\begin{equation} \label{eqchoitr2}
 (\tr_2 A^{\tau})\otimes I_n\ge \pm A^{\tau} .
\end{equation}

\item[(R4)] 
Ando  \cite{Ando14} (or see \cite{Lin16} for an alternative proof) revealed a nice connection between the first partial trace 
and the second partial trace, and established  
\begin{equation} \label{eqando}
 (\tr A)I_{mn} - (\mathrm{tr}_2 A) \otimes I_n \ge  
I_m\otimes (\mathrm{tr}_1 A)   -  A.
\end{equation}

\item[(R5)] 
Motivated by  (\ref{eqr5}) and  (\ref{eqando}), 
Li, Liu and Huang \cite{HL20} (or see \cite{LP2020} for a unified treatment) proved recently an analogous complement, 
which states that 
\begin{equation} \label{eqhl}
 (\tr A)I_{mn}  - (\tr_2 A) \otimes I_n
\ge \pm  \bigl( I_m\otimes (\tr_1 A)  -A\bigr),
\end{equation}
and Li et al. also obtained  
\begin{equation} \label{eqq4}
 (\tr A)I_{mn}  \pm (\tr_2 A) \otimes I_n \ge A \pm I_m\otimes (\tr_1 A).
\end{equation}
\end{itemize}

We will present some partial trace inequalities 
for positive semidefinite block matrices which 
improve the abovementioned results. 
The paper is organized as follows. 
In Section \ref{sec2}, we shall give a new method to prove Choi's result (R1). 
This  method can allow us to give an improvement on 
 (\ref{eqr5}); see Theorem  \ref{thm22}. 
In Section \ref{sec3}, we present inequalities 
for the second partial trace 
and show an improvement on  (\ref{eqchoitr2}); 
see Theorem \ref{thm24}.  
In Section \ref{sec4}, we shall prove some 
inequalities involving both the first and second partial trace. 
Our results  can  be viewed as  improvements on  
(\ref{eqhl}) and (\ref{eqq4}); see Theorems \ref{thm42}, \ref{thm44} 
and \ref{thm4.4}. 
In Section \ref{sec5},  we give  an application on Cauchy--Khinchin's 
inequality by using 
Theorems \ref{thm42} and 
\ref{thm44}; see Corollary \ref{ineqthm52}. 
In  Section \ref{sec6}, we study the majorization inequalities 
related to partial traces for two by two block matrices; see Theorem  \ref{thm52}. In addition, 
we also prove some inequalities about the unitarily invariant norms
 as well as the singular values, which extend slightly 
 two recent results of Lin \cite{Lin14} and \cite{Lin16b}; 
 see Theorems \ref{thm54} and \ref{thm37}.

\section{Inequalities about the first partial trace}
\label{sec2}

We shall give a short  proof 
 of Choi's result (R1). 
The original proof stated in \cite[Theorem 2]{Choi17} 
 used a standard decomposition 
of positive semidefinite matrices and then applied inductive techniques. 
Our method is quite different and transparent.  
As an application of this method,  
we shall present an improvement on 
 (\ref{eqr5}) of Choi's result (R3).

\medskip 
\noindent 
{\bf Alternative proof of  (R1).}
Denote $D_A:=A_{1,1}\oplus A_{2,2} \oplus \cdots  \oplus A_{m,m}$. We are going to prove 
\[ I_m \otimes \tr_1 A^{\tau}  + (m-2)D_A
\ge A^{\tau} + (m-2)D_A , \]
which is the same as 
\begin{eqnarray}
&& \notag \begin{bmatrix}
(m\!-\!1)A_{1,1} +\sum\limits_{i\neq 1} A_{i,i} & 0 & \cdots & 0 \\
0 & (m\!-\!1)A_{2,2} +\sum\limits_{i\neq 2} A_{i,i} & \cdots & 0 \\
\vdots & \vdots &&\vdots \\
0 & 0 & \cdots &(m\!-\!1)A_{m,m} +\sum\limits_{i\neq m} A_{i,i}
\end{bmatrix} \\
&  \label{eqkey} \ge &
\begin{bmatrix}
(m-1)A_{1,1}  & A_{2,1} & \cdots & A_{m,1} \\
A_{1,2} & (m-1)A_{2,2}  & \cdots & A_{m,2} \\
\vdots & \vdots &&\vdots \\
A_{1,m} & A_{2,m} & \cdots &(m-1)A_{m,m}
\end{bmatrix}. 
\end{eqnarray}
It suffices to show that for every pair $(i,j)$ with $1\le i <j \le m$,  
\begin{equation} 
\label{eq6} \begin{bmatrix} 
       & \text{\small $i$-th} &        &  \text{\small $j$-th} &   \\[-0.3cm]
       & {\vdots} &        & {\vdots} &        \\
\cdots & A_{i,i}+A_{j,j} & \cdots & 0 & \cdots \\
        & \vdots &  & \vdots &       \\
\cdots & 0 & \cdots & A_{i,i}+A_{j,j} & \cdots \\
       & \vdots &        & \vdots &        
\end{bmatrix} 
\ge 
\begin{bmatrix} 
       & \vdots &        & \vdots &        \\
\cdots & A_{i,i} & \cdots & A_{j,i} & \cdots \\
        & \vdots &  & \vdots &       \\
\cdots & A_{i,j} & \cdots & A_{j,j} & \cdots \\
       & \vdots &        & \vdots &        
\end{bmatrix} .
\end{equation}
Indeed,  we can 
sum all inequalities (\ref{eq6}) over all $1\le i<j\le n$, which leads to the 
desired (\ref{eqkey}). 
Note that the omitted blocks in (\ref{eq6}) are zero matrices, so we need to prove 
\[ \begin{bmatrix}A_{i,i}+A_{j,j} & 0 \\ 
0 & A_{i,i}+A_{j,j}  \end{bmatrix} \ge 
 \begin{bmatrix}A_{i,i} & A_{j,i} \\ 
A_{i,j} & A_{j,j}  \end{bmatrix}.  \]
This inequality immediately holds by observing that 
\[  \begin{bmatrix}A_{j,j} & -A_{j,i} \\ 
-A_{i,j} & A_{i,i}  \end{bmatrix} = 
 \begin{bmatrix} 0 & -I_n \\ 
I_n & 0  \end{bmatrix} 
 \begin{bmatrix}A_{i,i} & A_{i,j} \\ 
A_{j,i} & A_{j,j}  \end{bmatrix} 
\begin{bmatrix} 0 & I_n \\ 
-I_n & 0  \end{bmatrix}.   \]
Hence, we complete the proof.

\medskip

As promised, we shall provide an improvement on (\ref{eqr5}) by using the above method.

\begin{theorem}  \label{thm22}
Let $A= [A_{i,j}]_{i,j=1}^m\in \mathbb{M}_m(\mathbb{M}_n)$ be positive semidefinite. Then 
\[ I_m \otimes \tr_1 A^{\tau} \ge -A^{\tau} +2D_A, \]
where $D_A=A_{1,1}\oplus A_{2,2} \oplus \cdots  \oplus A_{m,m}$. 
\end{theorem}

\begin{proof} 
We intend to prove 
\[ I_m \otimes \tr_1 A^{\tau} - mD_A \ge -A^{\tau} -(m-2)D_A. \]
This inequality can be written as 
\begin{eqnarray*}
& & \begin{bmatrix}
-(m\!-\!1)A_{1,1} +\sum\limits_{i\neq 1} A_{i,i} & 0 & \cdots & 0 \\
0 & -(m\!-\!1)A_{2,2} +\sum\limits_{i\neq 2} A_{i,i} & \cdots & 0 \\
\vdots & \vdots &&\vdots \\
0 & 0 & \cdots &-(m\!-\!1)A_{m,m} +\sum\limits_{i\neq m} A_{i,i}
\end{bmatrix} \\
& \ge  &
\begin{bmatrix}
-(m-1)A_{1,1}  & -A_{2,1} & \cdots & -A_{m,1} \\
-A_{1,2} & -(m-1)A_{2,2}  & \cdots & -A_{m,2} \\
\vdots & \vdots &&\vdots \\
-A_{1,m} & -A_{2,m} & \cdots &-(m-1)A_{m,m}
\end{bmatrix}.
\end{eqnarray*}
Using a similar treatment as to the above proof of (R1), it is sufficient to prove 
\[ \begin{bmatrix} -A_{i,i}+A_{j,j} & 0 \\ 
0 & -A_{j,j} + A_{i,i}  \end{bmatrix} \ge 
 \begin{bmatrix}-A_{i,i} & -A_{j,i} \\ 
-A_{i,j} & -A_{j,j}  \end{bmatrix}  \]
for every pair $(i,j)$ with $1\le i<j\le m$, 
which follows by noting that 
\[  \begin{bmatrix}A_{j,j} & A_{j,i} \\ 
A_{i,j} & A_{i,i}  \end{bmatrix} = 
 \begin{bmatrix} 0 & I_n \\ 
I_n & 0  \end{bmatrix} 
 \begin{bmatrix}A_{i,i} & A_{i,j} \\ 
A_{j,i} & A_{j,j}  \end{bmatrix} 
\begin{bmatrix} 0 & I_n \\ 
I_n & 0  \end{bmatrix}.   \]
Thus, this completes the proof. 
\end{proof}

Although Theorem \ref{thm22}  is an improvement of Choi's result (\ref{eqr5}), 
as pointed out by a referee, a concise proof of Theorem \ref{thm22} can also be given from Choi's proof in \cite{Choi18}.

\begin{corollary}
Let $A=[A_{i,j}]_{i,j=1}^m \in  \mathbb{M}_m(\mathbb{M}_n)$ be Hermitian. Then 
\[ (m-1)\lambda_{\max} (A)I_{mn}  
\ge  I_m\otimes \tr_1 A^{\tau}  - A^{\tau} \ge (m-1)\lambda_{\min} (A)I_{mn} ,\]
and 
\[ (m-1)\lambda_{\max} (A)I_{mn} +2D_A
\ge  I_m\otimes \tr_1 A^{\tau} + A^{\tau} \ge (m-1)\lambda_{\min} (A)I_{mn} +2D_A,  \]
where $D_A=A_{1,1}\oplus A_{2,2} \oplus \cdots  \oplus A_{m,m}$. 
\end{corollary} 

\begin{proof}
The required result holds from  Choi's result (R1) and Theorem \ref{thm22} by replacing $A$ 
with $A-\lambda_{\min}(A)I_{mn}$ and $\lambda_{\max}(A)I_{mn}-A$. 
We leave the detailed proof to the readers.  
\end{proof}

\section{Inequalities about the second partial trace}
\label{sec3}

We  present some inequalities of the second partial trace in this section. 
First of all, we shall give an improvement on  (\ref{eqchoitr2}). 
To illustrate the relations between the first and second partial traces, 
we shall apply a useful technique, which was recently introduced by Choi \cite{Choi18}. 
Assume that $A=[A_{i,j}]_{i,j=1}^m\in \mathbb{M}_m(\mathbb{M}_n)$, where
$A_{i,j}= [a_{r,s}^{i,j} ]_{r,s=1}^n$. We define 
$B_{r,s}:=[a_{r,s}^{i,j}]_{i,j=1}^m\in \mathbb{M}_m$ and  
\begin{equation} \label{eqtilde} 
 \widetilde{A} := [B_{r,s}]_{r,s=1}^n  \in \mathbb{M}_n (\mathbb{M}_m). 
 \end{equation} 
 Clearly, both $A$ and $\widetilde{A}$ are matrices of order $mn\times mn$. 
Furthermore,  
$A$ and $\widetilde{A}$ are unitarily similar; see, e.g., \cite{Choi532} or \cite{Li20}. 
The similarity implies that $\widetilde{A}$ is  positive semidefinite whenever $A$ is positive semidefinite.  
 Next, we  make a brief review of some useful properties. 
 
 \begin{lemma} \label{lem-property} \cite{Choi18}
 Each of the following holds. \\
 (a) For $A,B\in \mathbb{M}_m(\mathbb{M}_n)$, 
 $A\le B$ implies $\widetilde{A} \le \widetilde{B}$. \\ 
 (b) For $X\in \mathbb{M}_m$ and $Y\in \mathbb{M}_n$, 
$ \widetilde{X\otimes Y} = Y\otimes X$. \\ 
 (c) $ \widetilde{A}^{\tau} = (\widetilde{A^{\tau}})^T= 
\widetilde{(A^{\tau})^T}$. \\ 
(d) For $A\in \mathbb{M}_m(\mathbb{M}_n)$, 
$ \mathrm{tr}_2 A  = \mathrm{tr}_1\widetilde{A}$ 
and $\mathrm{tr}_2\widetilde{A}^{\tau} = \mathrm{tr}_1A^T$. 
 \end{lemma}
 
 In the sequel, we write $A\circ B$ for the Hadamard (Schur) product of $A$ and $B$. 

\begin{theorem}  \label{thm24}
Let $A \in \mathbb{M}_m(\mathbb{M}_n)$ be positive semidefinite. Then  
\[ (\tr_2 A^{\tau}) \otimes I_n \ge - A^{\tau}+ 2A^{\tau} \circ J, \]
where $J$ is the $m\times m$ block matrix with each block $I_n$. 
\end{theorem}

\begin{proof}
Replacing $A$ with $\widetilde{A}$ in Theorem \ref{thm22}, we have  
\[  I_n \otimes \mathrm{tr}_1 \widetilde{A}^{\tau} 
\ge -\widetilde{A}^{\tau} +2D_{\widetilde{A}}.\]  
Applying Lemma \ref{lem-property}, we  get 
\begin{equation*} 
I_n \otimes \mathrm{tr}_1 \widetilde{(A^{\tau})^T}  
\ge 
-\widetilde{(A^{\tau})^T} + 2 {D_{\widetilde{A}}}. 
\end{equation*}
Observe that $X\ge Y$ implies $\widetilde{X}\ge \widetilde{Y}$. 
 Invoking Lemma \ref{lem-property} again, we obtain 
\begin{equation}  \label{eqtrans}
\left( \tr_2 (A^{\tau})^T\right) \otimes I_n = 
\tr_1 \widetilde{(A^{\tau})^T}  \otimes I_n 
\ge 
-(A^{\tau})^T + 2\widetilde{D_{\widetilde{A}}}. 
\end{equation}
Recall in (\ref{eqtilde}) that 
$B_{r,s}=[a_{r,s}^{i,j}]_{i,j=1}^m\in \mathbb{M}_m$. 
A direct computation reveals that 
\[  \widetilde{D_{\widetilde{A}}} = 
\widetilde{ \mathrm{diag}(B_{1,1}},B_{2,2},\ldots ,B_{n,n}) = [A_{i,j}\circ I_n]_{i,j=1}^m. \]
Note that $(X\otimes Y)^T = X^T\otimes Y^T$ and 
\[ (\widetilde{D_{\widetilde{A}}})^T=\left[ (A_{j,i}\circ I_n)^T\right]_{i,j=1}^m 
= [A_{j,i}\circ I_n]_{i,j=1}^m= A^{\tau} \circ J. \]
Taking transpose in both sides of (\ref{eqtrans}) yields the required result. 
\end{proof}

We replace $A$ 
with $A-\lambda_{\min}(A)I_{mn}$ and $\lambda_{\max}(A)I_{mn}-A$ in  
Theorem \ref{thm24}, which yields the following corollary. 

\begin{corollary}
Let $A\in  \mathbb{M}_m(\mathbb{M}_n)$ be Hermitian. Then 
\[ (n-1)\lambda_{\max} (A)I_{mn} + 2 A^{\tau} \circ J
\ge (\tr_2 A^{\tau})\otimes I_n + A^{\tau} \ge (n-1)\lambda_{\min} (A)I_{mn} + 
2A^{\tau}\circ J,  \]
where $J$ is the $m\times m$ block matrix with each block $I_n$. 
\end{corollary}

\section{Inequalities about two partial traces}
\label{sec4}

In this section, we shall present inequalities involving both 
the first and second partial traces, which are improvements 
on Li--Liu--Huang's result (R5).  
Recall that a map $\Phi : \mathbb{M}_n \to \mathbb{M}_k$ is called positive 
if it maps positive semidefinte matrices to positive semidefinite matrices. 
A map $\Phi :\mathbb{M}_n \to \mathbb{M}_k$ is said to be $m$-positive 
if for $[A_{i,j}]_{i,j=1}^m\in \mathbb{M}_m(\mathbb{M}_n)$, 
\[  [A_{i,j}]_{i,j=1}^m \ge 0 \quad \Rightarrow \quad [\Phi (A_{i,j})]_{i,j=1}^m \ge 0.   \]
We say that $\Phi$ is {\it completely positive} if it is $m$-positive for every integer $m\ge 1$. 
On the other hand, a map $\Phi : \mathbb{M}_n \to \mathbb{M}_k$ is called $m$-copositive if
 \[  [A_{i,j}]_{i,j=1}^m \ge 0 \quad \Rightarrow \quad [\Phi (A_{j,i})]_{i,j=1}^m \ge 0.   \]
Similarly, the map $\Phi$ is {\it completely copositive} if it is $m$-copositive for all $m\ge 1$. 
Furthermore, 
$\Phi$ is called completely PPT if the block matrix $[\Phi (A_{i,j})]_{i,j=1}^m$ is PPT for all $m\ge 1$  
whenever $[A_{ij}]_{i,j=1}^m\ge 0$, i.e., $\Phi$ is both  completely positive and completely copositive. 
It is well known that both the determinant map and the trace map are completely PPT; 
see \cite[p. 221, 237]{Zhang11} and   
see \cite[Chapter 3]{Bha07} for more standard results of completely positive maps. 

In 2014, Lin \cite{Lin14} proved that $\Phi (X)= (\tr X)I+X$ is a completely PPT map. 
Two years later,  Lin \cite{Lin16}  obtained that $\Psi (X)=(\tr X)I-X$ is a completely copositive map; see \cite{Li20laa} for related applications. 
It is easy to see that the completely copositivity of 
these two maps can also be deduced from  
Choi's result (\ref{eqchoitr2}).  
In this section,  we shall obtain some new matrix inequalities involving 
both the first and the second partial trace. 
Our results are  improvements on  (\ref{eqhl}) and (\ref{eqq4}).

It was proved in \cite[Example 1]{Bes13}   that if 
$\begin{bmatrix}A & B \\ B^* & C  \end{bmatrix}$ is positive semidefinite, then 
\begin{equation} \label{eqtr1} 
 \mathrm{tr} A\mathrm{tr} C -|\mathrm{tr} B|^2 \ge  \mathrm{tr}(AC) - \mathrm{tr}(B^*B) .  
\end{equation} 
We remark that (\ref{eqtr1}) was also  formulated in \cite{Zhang13,Lin14}. 
Note that  the positivity of the block matrix entails $\tr A \tr C - |\tr B|^2 \ge 0$; see \cite[p. 237]{Zhang11} 
or \cite[Theorem IX.5.10]{Bha97}. 
However,  the right hand side of 
(\ref{eqtr1}) might be negative.  
It has been shown that
the PPT condition on the block matrix can ensure $\tr AC \ge \tr B^*B$; see \cite[Theorem 2.1]{Lin15om}. 
Motivated by this observation, Kittaneh and Lin \cite{Lin17} 
further generalized (\ref{eqtr1}) 
by an elegant self-improved technique:  
\begin{equation}   \label{eqtr2}
 \mathrm{tr} A\mathrm{tr} C -|\mathrm{tr} B|^2 \ge \mathrm{tr}(B^*B) - \mathrm{tr}(AC).  
\end{equation}
By applying the $2$-copositivity of $\Psi (X)= (\tr X) I -X$, they  proved also that   
\begin{align}  \label{eqtr3}
 \mathrm{tr} A\mathrm{tr} C +|\mathrm{tr} B|^2 \ge  \mathrm{tr}(AC) +\mathrm{tr}(B^*B) . 
\end{align}
It is worth noting that the $2$-copositivity of $\Psi$ 
could also lead to the inequality (\ref{eqtr2}); 
see \cite{Li20laa} for more details.  
Let $A\in \mathbb{M}_m(\mathbb{M}_n) $ be positive semidefinite.  
By  using the $2$-copositivity of $\Psi (X)=(\tr X)I-X$, 
one can obtain Ando's result (\ref{eqando}). 
By employing the $2$-copositivity of $\Phi (X)=(\tr X)I+X$, 
Li, Liu and Huang \cite{HL20} established  (\ref{eqhl}) and (\ref{eqq4}).  
Last but not least, one may observe an interesting phenomenon 
that  (\ref{eqhl}) is similar to  
(\ref{eqtr1}) and (\ref{eqtr2}) in mathematical writing form. 
Correspondingly, (\ref{eqq4}) is also similar to (\ref{eqtr2}) and (\ref{eqtr3}).

In what follows, we will give an improvement on (\ref{eqq4}).

\begin{theorem} \label{thm42}
Let $A =[A_{i,j}]_{i,j=1}^m \in \mathbb{M}_m(\mathbb{M}_n)$ be positive semidefinite. Then  
\[  (\tr A)I_{mn}  + (\tr_2 A) \otimes I_n
\ge A+  I_m\otimes (\tr_1 A)  +  2(\tr_2 D_A)\otimes I_n - 2D_A,  \]
where $D_A=A_{1,1}\oplus A_{2,2} \oplus \cdots  \oplus A_{m,m}$.  
\end{theorem}

\begin{proof}
The desired inequality can be written as 
\begin{equation*}
\begin{bmatrix} 
\sum\limits_{i\neq 1} (\tr A_{i,i})I_n  & (\tr A_{1,2})I_n & \cdots & (\tr A_{1,m})I_n \\ 
(\tr A_{2,1}) I_n & \sum\limits_{i\neq 2} (\tr A_{i,i})I_n & \cdots & (\tr A_{2,m})I_n \\
\vdots & \vdots &  & \vdots \\
(\tr A_{m,1})I_n & (\tr A_{m,2}) I_n & \cdots & \sum\limits_{i\neq m}(\tr A_{i,i})I_n 
\end{bmatrix}  \ge  \begin{bmatrix} 
\sum\limits_{i\neq 1} A_{i,i} & A_{1,2} & \cdots & A_{1,m}  \\
A_{2,1} & \sum\limits_{i\neq 2} A_{i,i} & \cdots & A_{2,m}  \\
\vdots & \vdots & & \vdots \\
A_{m,1} & A_{m,2} & \cdots & \sum\limits_{i\neq m}A_{i,i}
\end{bmatrix}. 
\end{equation*}
It is easy to see from Theorem \ref{thm22} that 
\[  \begin{bmatrix} 
\sum\limits_{i\neq 1} A_{i,i} & A_{2,1} & \cdots & A_{m,1}  \\
A_{1,2} & \sum\limits_{i\neq 2} A_{i,i} & \cdots & A_{m,2}  \\
\vdots & \vdots & & \vdots \\
A_{1,m} & A_{2,m} & \cdots & \sum\limits_{i\neq m}A_{i,i}
\end{bmatrix} \ge 0. \]
Since $\Psi (X)=(\tr X)I-X$ is completely copositive. 
Then applying $\Psi$ to the above positive semidefinite block matrix yields the required inequality. 
\end{proof}

Note that $(\tr A_{i,i})I_n \ge A_{i,i}$ for each integer $i$, 
then $(\tr_2 D_A) \otimes I_n \ge D_A$, i.e., 
\[ 2(\tr_2 D_A)\otimes I_n - 2D_A \ge 0.  \]
So Theorem \ref{thm42} is indeed a generalization of (\ref{eqq4}). 
In view of symmetry of definitions of $\tr_1$ and $\tr_2$, 
one can easily obtain the following equivalent theorem.

\begin{theorem} \label{thm44}
Let $A\in \mathbb{M}_m(\mathbb{M}_n)$ be positive semidefinite. Then
\begin{equation*} \label{eq7}
 (\tr A)I_{mn}  - (\tr_2 A) \otimes I_n
\ge A- I_m\otimes (\tr_1 A) +  2(I_m\otimes \tr_1 A- A)\circ J, 
\end{equation*}
where $J$ is the $m\times m$ block matrix with each block $I_n$.
\end{theorem}

The next result is an analogue of Theorem \ref{thm42}.

\begin{theorem} \label{thm4.4}
Let $A=[A_{i,j}]_{i,j=1}^m \in \mathbb{M}_m(\mathbb{M}_n)$ be positive semidefinite. Then  
\[  (\tr A)I_{mn}  + (\tr_2 A) \otimes I_n + I_m\otimes (\tr_1 A)  +A
\ge  2(\tr_2 D_A)\otimes I_n + 2D_A,  \]
where $D_A=A_{1,1}\oplus A_{2,2} \oplus \cdots  \oplus A_{m,m}$.  
\end{theorem}

We remark that Theorems \ref{thm42}, \ref{thm44} and \ref{thm4.4} 
all require the positivity of $A$, by replacing $A$ by $A -\lambda_{\min}(A)I_{mn}$ and $\lambda_{\max}(A)I_{mn}- A$, one can extend 
these theorems to Hermitian matrices. 
Motivated by Theorems \ref{thm42} and  \ref{thm44}, 
we believe highly that it is possible to 
improve Ando's result (\ref{eqando}). 
In other words, 
for every positive semidefinite matrix $A\in \mathbb{M}_m(\mathbb{M}_n)$,  
is there a positive semidefinite matrix $T$  satisfying the following inequality? 
\begin{equation*} \label{eq}
 (\tr A)I_{mn}  +A
\ge  I_m\otimes (\tr_1 A) + (\tr_2 A) \otimes I_n + T. 
\end{equation*}

\section{Extensions on Cauchy--Khinchin's inequality} 
\label{sec5}

In this section, 
we shall provide some applications of 
Theorem \ref{thm42} and Theorem \ref{thm44} 
in the field of numerical inequalities. 
The Cauchy--Khinchin  inequality (see \cite[Theorem 1]{van98}) 
 states that if $X=[x_{ij}]$ is an  $m\times n$ real matrix, then 
\begin{equation} \label{eqck} 
 \left(\sum_{i=1}^m\sum_{j=1}^nx_{ij}\right)^2+mn\sum_{i=1}^m\sum_{j=1}^nx_{ij}^2\ge m\sum_{i=1}^m\left(\sum_{j=1}^nx_{ij}\right)^2+n\sum_{j=1}^n\left(\sum_{i=1}^mx_{ij}\right)^2. 
\end{equation}
It was proved that this inequality has many applications 
 on problems related to the directed graph in combinatorics; 
 see, e.g., \cite{deC1998,LP2001} for more details. 
In 2016, Lin \cite{Lin16} provided 
a simple proof using Ando's inequality (\ref{eqando}).  
Under the similar line, Li, Liu and Huang \cite{HL20} 
proved 
some generalizations by applying (\ref{eqhl}) and (\ref{eqq4}).  
In this section, we shall prove 
the following corollary, which is an easy consequence of Theorem \ref{thm42}.

  	\begin{corollary}  \label{ineqthm52}
 Let $X=[x_{ij}]$ be an $m\times n$ real matrix. Then 
\begin{eqnarray*}    
(m-2)n\sum_{i=1}^m\sum_{j=1}^nx_{ij}^2 + n\sum_{j=1}^n\left(\sum_{i=1}^mx_{ij}\right)^2
\ge     
\left(\sum_{i=1}^m\sum_{j=1}^nx_{ij}\right)^2 +(m-2) \sum_{i=1}^m\left(\sum_{j=1}^nx_{ij}\right)^2,
\end{eqnarray*}
		and 
\begin{eqnarray*}    
m(n-2)\sum_{i=1}^m\sum_{j=1}^nx_{ij}^2 + 
m \sum_{i=1}^m\left(\sum_{j=1}^nx_{ij}\right)^2 
\ge     
\left(\sum_{i=1}^m\sum_{j=1}^nx_{ij}\right)^2 + 
(n-2)\sum_{j=1}^n\left(\sum_{i=1}^mx_{ij}\right)^2. 
  		\end{eqnarray*}
\end{corollary}

 	\begin{proof}  
	We only prove the first inequality since the second is similar. 
	Let $J_n$ be an $n$-square matrix with all entries 1.  
	Setting $A =J_m\otimes J_n$ in Theorem \ref{thm42}. 
	Clearly, $\tr_1 A = mJ_n $ and $ \tr_2 A = n J_m$. 
	Additionally,  we have $D_A=J_n \oplus \cdots \oplus J_n$, 
	which leads to $(\tr_2 D_A ) \otimes I_n =(nI_m) \otimes I_n$. 
Therefore, 	
\begin{equation} \label{eqeq18} (m-2)nI_{mn} + nJ_m\otimes I_n \ge 
 J_m\otimes J_n + (m-2) I_m\otimes J_n. 
 \end{equation}
Let  $\mathrm{vec}\, X=[x_{11}, \ldots, x_{1n}, x_{21}, \ldots, x_{2n}, \ldots, x_{m1}, \ldots, x_{mn}]^T \in \mathbb{R}^{mn}$ be the column vector determined by matrix $X$. 
Then some calculations give the following equalities: 
  			\begin{eqnarray*} 
  	 (\mathrm{vec}\, X)^TI_{mn}(\mathrm{vec}\, X) &=& \sum_{i=1}^m\sum_{j=1}^nx_{ij}^2,\\  	
(\mathrm{vec}\, X)^T(J_m\otimes I_n)(\mathrm{vec}\, X) 
   &=&\sum_{j=1}^n\left(\sum_{i=1}^mx_{ij}\right)^2, \\ 	
   (\mathrm{vec}\, X)^T (J_m\otimes J_n)(\mathrm{vec}\, X) 
&=& \left(\sum_{i=1}^m\sum_{j=1}^nx_{ij}\right)^2, \\ 
  	 (\mathrm{vec}\, X)^T(I_m\otimes J_n)(\mathrm{vec}\, X)
   &=&\sum_{i=1}^m\left(\sum_{j=1}^nx_{ij}\right)^2. 
  			\end{eqnarray*}
  		Thus the  desired inequality is equivalent to
  	\begin{equation*} \begin{aligned}  
&(\mathrm{vec}\, X)^T ( (m-2)nI_{mn} + nJ_m\otimes I_n  )(\mathrm{vec}\, X) \\
  &\quad \ge 
 (\mathrm{vec}\, X)^T ( J_m\otimes J_n + (m-2) I_m\otimes J_n )(\mathrm{vec}\, X) .
  		\end{aligned} \end{equation*}   
This inequality  follows immediately from the matrix inequality (\ref{eqeq18}). 
 \end{proof}

\noindent 
{\bf Remark.} 
The inequality (\ref{eqeq18}) can also be proved 
by a standard argument on eigenvalues by noting that 
$J_{mn},I_{mn},I_m \otimes J_n$ and 
$J_m \otimes I_n$ mutually commute. 
In addition, setting $A=J_m \otimes J_n$, we see that 
$A=A^{\tau}$  and  $A$ is PPT. 
Unfortunately, our results in Sections \ref{sec2} and \ref{sec3}, e.g., Theorems  \ref{thm22} and \ref{thm24},  will 
lead to some trivial  inequalities.

\section{More inequalities for two by two block matrices} 
\label{sec6}

In this section, 
we will investigate the $2\times 2$ block positive semidefinite matrices. 
First of all, we introduce the notion of majorization. 
For a vector $\bm{x}=(x_1,\ldots ,x_n)\in \mathbb{R}^n$, 
we arrange the coordinates of $\bm{x}$ 
in non-increasing order $x_1^{\downarrow}\ge \cdots \ge x_n^{\downarrow}$ and denote 
$\bm{x}^{\downarrow}=(x_1^{\downarrow},\ldots ,x_n^{\downarrow})$. 
Given $\bm{x},\bm{y}\in \mathbb{R}^n$, we say that $\bm{x}$ is {\it weakly majorized} by $\bm{y}$, 
written as $\bm{x} \prec_w \bm{y}$, if 
\[ \sum_{i=1}^k x_i^{\downarrow} \le \sum_{i=1}^k y_i^{\downarrow} \quad 
\text{for $k=1,2,\ldots ,n$.} \]
We say  that $\bm{x}$ is {\it majorized} by $\bm{y}$, denoted by $\bm{x} \prec \bm{y}$, 
if $\bm{x}\prec_w \bm{y}$ and the sum of all entries of $\bm{x}$ equal to the sum of all 
entries of $\bm{y}$. 
More generally, 
if the dimension of $\bm{x}$ is larger than $\bm{y}$, the inequality 
$\bm{x}\prec \bm{y}$ really means that $\bm{x}\prec (\bm{y},\bm{0})$, 
where the zero vector is added to make the length of $(\bm{y},\bm{0})$ 
the same as that of $\bm{x}$. 
The majorization theory has become a rich research field 
with far-reaching applications to a wide number of areas, 
we refer to the recent monograph \cite{MOA11} or 
\cite[Chapter 10]{Zhang11}
for  comprehensive surveys on this subject.

Let $H$ be a Hermitian matrix. We denote by $\bm{\lambda}(H)$ 
the vector of eigenvalues of $H$ in which the components are 
sorted in decreasing order. 
Historically, the first result of majorization arising in matrix theory is 
usually attributed to Issai Schur, who proved that 
the diagonal elements of  $H$
are majorized by its  eigenvalues, i.e., $\bm{d}(H)\prec \bm{\lambda} {(H)}$. 
This  majorization provided a new and profound understanding on 
Hadamard's determinant inequality; see, e.g., \cite[p. 514]{HJ13}. 
Due to Schur's discovery, a large number of majorization inequalities have been found in the context of matrix analysis. 
For instance,  
if $A=\left[ A_{i,j}\right]_{i,j=1}^m \in \mathbb{M}_m(\mathbb{M}_n)$ is Hermitian, then Schur's inequality implies  
$ \bm{d} (A)\prec  \bm{\lambda} (A_{1,1} \oplus \cdots \oplus A_{m,m})$. 
Furthermore, 
if $A$ is positive semidefinite, 
it is well-known (see \cite[p. 259]{HJ13} and \cite[p. 308]{MOA11}) that 
 \begin{equation} \label{eqm1}
 \bm{\lambda} (A_{1,1} \oplus \cdots \oplus A_{m,m})\prec \bm{\lambda} (A) \prec 
  \bm{\lambda} (A_{1,1}) + \cdots + \bm{\lambda} (A_{m,m}). 
\end{equation}
Moreover, Rotfeld and Thompson \cite[p. 330]{MOA11} obtained an analogous complement, 
\begin{equation} \label{eqm2}
  \bm{\lambda} (A_{1,1} \oplus \cdots \oplus A_{m,m}) 
  \prec \bm{\lambda} (\tr_1 A) \prec 
  \bm{\lambda} (A_{1,1}) + \cdots + \bm{\lambda} (A_{m,m}). 
\end{equation}

In the sequel, we shall provide a comparison between (\ref{eqm1}) and (\ref{eqm2}) under the PPT condition. 
First of all, we recall the following lemma, which was proved 
by Hiroshima \cite{Hiro03} in the language of quantum information theory; see \cite{Lin15laa} for an alternative proof.

\begin{lemma}  \cite{Hiro03,Lin15laa}  \label{lemHiro}
Let $A\in \mathbb{M}_m(\mathbb{M}_n)$ be positive semidefinite. 
\begin{itemize}
\item[(1)] If $I_m \otimes \tr_1 A \ge A$, 
then $\bm{\lambda} (A)\prec \bm{\lambda} (\tr_1 A)$.

\item[(2)] If $(\tr_2 A)\otimes I_n \ge A$, 
then $\bm{\lambda} (A)\prec \bm{\lambda} (\tr_2 A)$.
\end{itemize}
\end{lemma}

Our first result in this section is the following majorization inequalities, 
which is a direct consequence of  Lemma \ref{lemHiro} 
by applying inequalities (\ref{eqr5}) and (\ref{eqchoitr2}). 

\begin{theorem}  \label{thm52}
Let  $A \in  \mathbb{M}_m(\mathbb{M}_n)$ be PPT.  Then
 \[  \bm{\lambda} (A)\prec \bm{\lambda} (\tr_1 A) ~~\text{and}~~  \bm{\lambda} (A)\prec \bm{\lambda} (\tr_2 A). \]
Similarly, we have 
\[  \bm{\lambda} (A^{\tau})\prec \bm{\lambda} (\tr_1 A) ~~\text{and}~~ 
\bm{\lambda} (A^{\tau})\prec \bm{\lambda} (\tr_2 A). \]
\end{theorem}

In particular, for the case of $2\times 2$ block matrices, 
we remark here that  inequalities involving $\mathrm{tr}_1$ in 
Theorem \ref{thm52} was partially 
proved by Bourin, Lee and Lin \cite{BLL2012,BLL2013} 
by making use of a simple but useful decomposition lemma 
for positive semidefinite matrices. 
Moreover, other special cases can be found in \cite{LW12} and  \cite[Corollary 5]{TPZ12}.

Over the past few years, 
$2\times 2$ block positive partial transpose matrices play an important role in matrix analysis 
and quantum information, 
such as the separability of mixed states and the subadditivity of $q$-entropies; 
see \cite{Bes13, Lin15om} for related topics and references to the physics literature. 
It is extremely meaningful and significant 
to finding $2\times 2$ block PPT matrices. 
To the author’s best knowledge,  the most famous example is commonly regarded as the Hua matrix; see \cite{Hua1963,Ando2008,XXZ2011,Lin15om,Lin2016laa} for more details.

In 2014, Lin \cite[Proposition 2.2]{Lin14} proved that if 
$\begin{bmatrix}A & B \\ B^* & C \end{bmatrix} 
\in \mathbb{M}_2(\mathbb{M}_n)$ is positive semidefinite, then 
\[ \begin{bmatrix} 
(\tr A) I + A & (\tr B)I + B \\ 
(\tr B^*)I + B^*  & (\tr C)I+C 
 \end{bmatrix}  \]
 is PPT.  
 In 2017, Choi \cite[Theorem 4]{Choi17} also showed an extremely similar result, 
 which states  that 
 \begin{equation} \label{choippt}
 \begin{bmatrix} (\tr A) I +C & (\tr B)I-B \\ (\tr B^*)I-B^* & (\tr C)I+A \end{bmatrix} 
 \end{equation}
 is positive semidefinite. 
In the next theorem, 
we shall prove that the $2\times 2$ block matrix in (\ref{choippt}) 
is further  PPT.

\begin{theorem} \label{thm54}
Let $\begin{bmatrix}A & B \\ B^* & C \end{bmatrix}\in \mathbb{M}_2(\mathbb{M}_n)$ 
be positive semidefinite. Then 
\[   \begin{bmatrix} (\tr A) I +C & (\tr B)I-B \\ (\tr B^*)I-B^* & (\tr C)I+A \end{bmatrix} \]
is PPT. 
\end{theorem}

\begin{proof}
In view of the positivity of (\ref{choippt}), 
we only need to prove 
\[  \begin{bmatrix} (\tr A) I +C & (\tr B^*)I-B^* \\ (\tr B)I-B & (\tr C)I+A \end{bmatrix}\ge 0. \]
Note that  
\[ \begin{bmatrix}C & -B^* \\ -B & A \end{bmatrix} = 
\begin{bmatrix}0 & -I \\ I & 0 \end{bmatrix}
\begin{bmatrix}A & B \\ B^* & C \end{bmatrix}
\begin{bmatrix}0 & I \\ -I & 0 \end{bmatrix}\ge 0. \]
It suffices to show 
\[ \begin{bmatrix}(\tr A)I & (\tr B^*)I \\ (\tr B)I & (\tr C)I \end{bmatrix}\ge 0, \]
which follows from the complete positivity of trace map.  
\end{proof}

A norm $\lVert \cdot \rVert$ on $\mathbb{M}_n$ is called 
unitarily invariant if 
$\lVert UAV \rVert = \lVert A \rVert$ 
for any $A\in \mathbb{M}_n$ 
and any unitary matrices $U,V\in \mathbb{M}_n$. 
The unitarily invariant norm of a matrix is closely related to  
its singular values; see, e.g., \cite[pp. 372--376]{Zhang11} and 
\cite[pp. 91--98]{Bha97}.  
Recall in Section \ref{sec4} that 
$\Phi (X)=(\tr X)I+X$ and $\Psi (X)=(\tr X)I-X$. 
In the sequel, we shall present some  inequalities involving 
unitarily invariant norms and singular values for these two maps.

\begin{corollary} \label{coro55}
Let $\begin{bmatrix}A & B \\ B^* & C \end{bmatrix} 
\in \mathbb{M}_2(\mathbb{M}_n) $ 
be positive semidefinite. Then
\[ 2 \big\lVert   \Phi (B) \big\rVert \le \left\lVert   \Phi (A)+ \Phi (C) \right\rVert  \] 
and 
\[ 2 \big\lVert  \Psi (B) \big\rVert \le \left\lVert  \Phi (A)+ \Phi(C) \right\rVert  \] 
for any unitarily invariant norm. 
\end{corollary}

\begin{proof} 
Invoking a  theorem of von Neumann (see \cite[p. 375]{Zhang11}), it is sufficient to prove 
\[  2\bm{s}((\tr B)I \pm  B) \prec_{w} 
\bm{s}(\bigl( \tr (A+C)\bigr)I + A+C),  \]
where $\bm{s}(X)$ is the vector consisting of singular values of $X$. 
We next prove the second inequality only. 
There is a well-known fact that if 
$X=\left[\begin{smallmatrix}Y & Z \\ Z^* & W \end{smallmatrix}\right]$ is positive semidefinite, then 
$ 2 s_i (Z) \le s_i (X) $ for every $i$. Thus 
\[   2\bm{s}((\tr B)I -  B) \prec_{w} 
\bm{s}   \begin{bmatrix} (\tr A) I +C & (\tr B)I-B \\ (\tr B^*)I-B^* & (\tr C)I+A \end{bmatrix}   \prec_{w} 
\bm{s}(\bigl( \tr (A+C)\bigr)I + A+C), 
 \]
where we used Theorem \ref{thm54} and Theorem \ref{thm52}. 
\end{proof}

The second partial trace inequality in Theorem \ref{thm52}
 yields the following corollary. 

\begin{corollary}
Let $H=\begin{bmatrix}A & B \\ B^* & C \end{bmatrix} 
\in \mathbb{M}_2(\mathbb{M}_n)$ 
be positive semidefinite. Then
\[   2\big\lVert \Phi (B)\big\rVert \le  (n+1)
\left\lVert  \mathrm{tr}_2 H  \right\rVert   \]
for any unitarily invariant norm. 
\end{corollary}

At the end of this paper, we conclude with the
 following  inequality involving singular values, 
 which is a stronger  inequality than Corollary \ref{coro55}.

\begin{theorem} \label{thm37}
Let $\begin{bmatrix}A & B \\ B^* & C \end{bmatrix} 
\in \mathbb{M}_2(\mathbb{M}_n)$ 
be positive semidefinite. Then for $j=1,2,\ldots, n$, 
\[ 2s_j \bigl( \Phi (B) \bigr) \le s_j \bigl( 
\Phi (A) + \Phi (C) \bigr), \]
and 
\[ 2s_j \bigl( \Psi (B) \bigr) \le s_j \bigl(   \Phi (A) + \Phi (C) \bigr). \]
\end{theorem}

 Note that the first inequality in 
Theorem \ref{thm37} 
was  proved by Lin as a main result in  \cite{Lin16b}. 
To some extent, 
our proof of Theorem \ref{thm37} grows out from  \cite{Lin16b}
with some simplifications.  
To proceed the proof, 
we need to present the following two lemmas.

\begin{lemma} \label{lem39}  \cite[p. 262]{Bha97} 
For any $M,N\in \mathbb{M}_{n\times m}$ and $j=1,2,\ldots ,n$, we have 
\[ 2s_j(MN^*) \le s_j(M^*M +N^*N). \]
\end{lemma}

\begin{lemma} \label{lem38}
For any $M,N\in \mathbb{M}_{n\times m}$ and $j=1,2,\ldots ,n$, we have 
\begin{align*}
\lambda_j (M^*M + N^* N) &\le \lambda_j (MM^* + NN^*) \\
&\quad + \frac{1}{2} \tr (M^*M +N^*N -M^*N-N^*M).
\end{align*}
\end{lemma}

\begin{proof}
By a direct computation, we can get 
\begin{align*}
2\lambda_j (MM^*+NN^*) 
&= \lambda_j \bigl( (M+N)(M+N)^* + (M-N)(M-N)^*\bigr) \\
&\ge \lambda_j \bigl( (M+N)(M+N)^*\bigr)+ \lambda_n \bigl( (M-N)(M-N)^*\bigr) \\
&\ge  \lambda_j \bigl( (M+N)^*(M+N)\bigr) \\
&= \lambda_j \bigl( 2(M^*M+N^*N)-(M-N)^*(M-N)\bigr) \\
&\ge 2\lambda_j \bigl( (M^*M+N^*N) \bigr) -\lambda_1 \bigl( (M-N)^*(M-N)\bigr) ,
\end{align*}
where the first and 
last inequality hold by Weyl's eigenvalue inequality (see \cite[p. 63]{Bha97}). 
Moreover, the positivity of $(M-N)^*(M-N)$ leads to 
\begin{align*} 
\lambda_1 \bigl( (M-N)^*(M-N)\bigr)&\le \tr \bigl( (M-N)^*(M-N)\bigr) \\
&=  \tr (M^*M +N^*N -M^*N-N^*M). 
\end{align*}
This completes the proof. 
\end{proof}

\noindent
{\bf Proof of Theorem \ref{thm37}}~~
Since $\begin{bmatrix}A & B \\ B^* & C \end{bmatrix}$ 
is positive semidefinite, we may write
\[ \begin{bmatrix}A & B \\ B^* & C \end{bmatrix} =
 \begin{bmatrix}MM^* & MN^* \\ NM^* & NN^*\end{bmatrix}   \]
for some $M,N\in \mathbb{M}_{n\times 2n}$. Then the desired inequality is the same as 
\[ 2s_j \bigl(  \tr (MN^*)I \pm  MN^*\bigr) 
\le s_j \bigl(  \tr (MM^*+NN^*)I + MM^*+NN^*\bigr).  \]
By the Weyl's inequality of singular value and Lemma \ref{lem39}, we have 
\begin{align*}
2s_j \bigl(  \tr (MN^*)I \pm  MN^*\bigr)  
&\le 2s_j (\pm MN^*) +2s_1\bigl( (\tr MN^*)I\bigr) \\
& = 2s_j (MN^*) +2 |\tr (MN^*)| \\
& \le s_j (M^*M + N^*N) +2 |\tr (MN^*)|, 
\end{align*}
On the other hand, we observe that   
\begin{align*} 
s_j \bigl(  \tr (MM^*+NN^*)I + MM^*+NN^*\bigr) =\lambda_j (MM^* + NN^*) +\tr (MM^*+NN^*). 
\end{align*}
It is sufficient to prove that for every $M,N\in \mathbb{M}_{n\times m}$,
\begin{equation} \label{eqqq}
s_j (M^*M + N^*N) +2 |\tr (MN^*)| \le \lambda_j (MM^* + NN^*) +\tr (MM^*+NN^*).
\end{equation}
We may assume without loss of generality that $\tr (MN^*)\ge 0$ in (\ref{eqqq}), 
since it is clear that (\ref{eqqq}) is invariant when we replace $M$ with $e^{i\theta}M$ 
for every $\theta \in [0,2\pi]$. Therefore,  
\begin{eqnarray}
& &s_j (M^*M + N^*N) + |\tr (MN^*)|  \notag  \\ 
&= &\lambda_j (M^*M + N^*N) + \frac{1}{2}\tr (M^*N+N^*M) \notag  \\
& \le & \lambda_j(MM^* + NN^*) + \frac{1}{2}\tr (M^*M + N^*N), \label{eqeq20}
\end{eqnarray}
where the last inequality holds by Lemma \ref{lem38}. 
Note that 
\[   \begin{bmatrix}\tr MM^* & \tr MN^* \\ \tr NM^* & \tr NN^*\end{bmatrix} \]
is a positive semidefinite matrix, we have 
\begin{equation}  \label{eqeq21}
 |\tr (MN^*)| = \frac{1}{2} \tr(MN^* + NM^*) 
\le \frac{1}{2}\tr (MM^* + NN^*). \end{equation}
The desired inequality (\ref{eqqq})  follows by combining  (\ref{eqeq20}) and (\ref{eqeq21}).   $\blacksquare$

\section*{Acknowledgments} 
The paper is dedicated to Prof. Weijun Liu, 
my teacher on the occasion of his 60th birthday, October 22 of the lunar calendar in 2021.   
This work was supported by  NSFC (Grant No. 11931002).  
The author would like to thank  Prof. Minghua Lin 
for bringing the topic on partial traces  
to his attention. 
Thanks also go to Prof. Fuzhen Zhang and Prof. Yuejian Peng 
for reading an earlier draft of the paper, 
Prof. Xiaohui Fu for the inspiring discussions over the years, 
and the anonymous referee for 
helpful suggestions on improving the presentation of this paper.

\end{document}